\UseRawInputEncoding 
\documentclass[12pt]{amsart}       
\usepackage{txfonts}
\usepackage{algorithm}
\usepackage{algorithmic}
\usepackage{amssymb}
\usepackage{eucal}
\usepackage{graphicx}
\usepackage{amsmath}
\usepackage{enumitem}
\usepackage{amscd}
\usepackage[all]{xy}           
\usepackage{tikz}
\usepackage{tikz-qtree}
\usepackage{amsfonts,latexsym}
\usepackage{xspace}
\usepackage{epsfig}
\usepackage{float}
\usepackage{color}
\usepackage{fancybox}
\usepackage{colordvi}
\usepackage{multicol}
\usepackage{colordvi}
\usepackage{wasysym}
\usepackage{makecell} 
\usepackage[numbers,sort&compress]{natbib}
\usepackage[active]{srcltx} 
\usepackage{CJK}
\ifpdf
  \usepackage[colorlinks,final,backref=page,hyperindex]{hyperref}
\else
  \usepackage[colorlinks,final,backref=page,hyperindex,hypertex]{hyperref}
\fi

\newcommand{\nc}{\newcommand}
\newcommand{\delete}[1]{}
\nc{\bfk}{{\bf k}}

\nc{\mlabel}[1]{\label{#1}}  
\nc{\mcite}[1]{\cite{#1}}  
\nc{\mref}[1]{\ref{#1}}  
\nc{\mbibitem}[1]{\bibitem{#1}} 

\delete{
\nc{\mlabel}[1]{\label{#1}  
{\hfill \hspace{1cm}{\small\tt{{\ }\hfill(#1)}}}}
\nc{\mcite}[1]{\cite{#1}{\small{\tt{{\ }(#1)}}}}  
\nc{\mref}[1]{\ref{#1}{{\tt{{\ }(#1)}}}}  
\nc{\mbibitem}[1]{\bibitem[\bf #1]{#1}} 
}

\newtheorem{theorem}{Theorem}[section]

\newtheorem{lemma}[theorem]{Lemma}
\newtheorem{problem}[theorem]{problem}

\theoremstyle{definition}
\newtheorem{definition}[theorem]{Definition}

\newtheorem{conjecture}{Conjecture}

\newtheorem{tempex}[theorem]{Example}
\newtheorem{tempexs}[theorem]{Examples}
\newtheorem{temprmk}[theorem]{Remark}
\newtheorem{tempexer}{Exercise}[section]

\nc{\tred}[1]{\textcolor{red}{#1}} \nc{\tgreen}[1]{\textcolor{green}{#1}}
\nc{\tblue}[1]{\textcolor{blue}{#1}} \nc{\tpurple}[1]{\textcolor{purple}{#1}}

\nc{\hu}[1]{\tpurple{\underline{Hu:}#1 }}
\nc{\xing}[1]{\tblue{\underline{Xing:}#1 }}

\nc{\GS}{Gr\"obner-Shirshov\xspace}
\nc{\gsb}{Gr\"{o}bner-Shirshov basis\xspace}
\nc{\gsbs}{Gr\"{o}bner-Shirshov bases\xspace}
\nc\olie{operated Lie algebra\xspace}
\nc\olies{operated Lie algebras\xspace}
\nc\bfone{\mathbf{1}}

\nc\nas[1]{{#1}^\ast}
\nc\alsw[1]{{\rm ALSW}(#1)}
\nc\nlsw[1]{{\rm NLSW}(#1)}
\nc\clie[1]{[#1]}
\nc\lbar[1]{\overline{#1}}
\nc\suba[1]{|_{#1}}
\nc\coplie[1]{\bfk\nlsbw{#1}}
\nc\coplieo[2]{\bfk{\rm NLSBW}_{#2}(#1)}
\nc\lb[1]{\left[#1\right]}\nc\dt[1]{{t^{(#1)}}}
\nc{\Irr}{\mathrm{Irr}}
\nc\blw[1]{\lfloor#1\rfloor}
\nc\plie[1]{\mathfrak{S}(#1)}
\nc\plien[1]{\mathcal{N}(#1)}
\nc{\lc}{\lfloor} \nc{\rc}{\rfloor}
\nc\Id{\rm Id}\nc\sopm[1]{\mathfrak{S}^\star(#1)}

\nc\ordc{>_{{\rm Dl}}} \nc\ordqc{\geq_{{\rm Dl}}}
\nc\ord{>_{{\rm IM }}}\nc\ordq{\geq _{\rm IM}}
\nc\ordd{>_{{\rm IM}}}\nc\ordqd{\geq_{{\rm IM}}}
\nc\ordb{>_{{\rm IM}}}\nc\ordqb{\geq_{{\rm IM}}}
\nc\alsbw[1]{{\rm ALSBW}_{\ordq}(#1)} \nc\nlsbw[1]{{\rm NLSBW}_{\ordq}(#1)}
\nc\alsbwo[2]{{\rm ALSBW}_{#2}(#1)} \nc\nlsbwo[2]{{\rm NLSBW}_{#2}(#1)}\nc\bws[1]{{\lfloor#1\rfloor}}\nc\oplie{{\rm OLie}(X)}
\nc{\dep}{{\rm dep}}
\nc\ordt{\geq_{\rm dt}}
\nc\ordtl{>_{\rm dt}}

\begin{document}

\title[On the nonnegativity of monomial immanants for hook partitions]{On the nonnegativity of monomial immanants for hook partitions}

\author{Xiangshuai Dong, Tingzeng Wu$^{*}$ and Xing Gao}\thanks{*Corresponding author}

\address{School of Mathematical Sciences, Xiamen University, Xiamen, 361005,  P.R.~China}
\email{a3566293588@163.com}

\address{School of Mathematics and Statistics, Qinghai Minzu University, Xining, Qinghai 810007, P.R.~China; Qinghai Institute of Applied Mathematics, Xining, Qinghai 810007, P.R.~China}
\email{mathtzwu@163.com}

\address{School of Mathematics and Statistics, Lanzhou University
Lanzhou, 730000, China;
Gansu Provincial Research Center for Basic Disciplines of Mathematics
and Statistics, Lanzhou, 730070, China
}
\email{gaoxing@lzu.edu.cn}

\date{\today}

\begin{abstract}
Let \(A=(a_{ij})\) be an \(n\times n\) real matrix and let \(\lambda\) be a partition of \(n\). Let \(\phi^\lambda\) be the class function dual to the Young permutation character, and let
\[
  \phi^\lambda[A]
  = \sum_{\sigma\in\mathfrak{S}_n}\phi^\lambda(\sigma)
    \prod_{i=1}^{n}a_{i\sigma(i)}
\]
be the corresponding monomial immanant. Stembridge [Canad. J. Math. 44 (1992), pp. 1079--1099] posed the following open problem: If all minors of \(A\) of order at most \(r\) are nonnegative, and the partition \(\lambda\) has length at most \(r\), is it true that \(\phi^\lambda[A]\ge 0\)? This paper solves the problem for all hook partitions.
\end{abstract}

\makeatletter
\@namedef{subjclassname@2020}{\textup{2020} Mathematics Subject Classification}
\makeatother
\subjclass[2020]{15A15; 05E05; 20C30; 15A45}

\keywords{Class functions; Monomial immanants; Hook partitions; Nonnegative matrices}

\maketitle

\tableofcontents

\setcounter{section}{0}

\allowdisplaybreaks
\section{Introduction}\label{sec:introduction}

\subsection{Immanants and total nonnegativity}\label{subsec:background}

Immanants connect matrix inequalities with the representation theory of
the symmetric group and the theory of symmetric functions. Their
classical study goes back to Schur and Littlewood
\cite{lit,sch}. Immanants have been extensively studied in many areas \cite{chan,don1,don2,yu}. For a real-valued class function $f$ on the symmetric
group $\mathfrak{S}_n$ and a real matrix $A=(a_{ij})_{i,j=1}^n$, the
associated \emph{generalized immanant} is
\[
 f[A]=\sum_{\sigma\in\mathfrak{S}_n}
 f(\sigma)\prod_{i=1}^n a_{i\sigma(i)}.
\]
The sign and trivial characters give the determinant and the permanent,
respectively. A central question is to identify class functions for
which this expression is nonnegative under natural assumptions on the
minors of $A$. Such questions also have symmetric-function counterparts
through immanants of Jacobi--Trudi matrices \cite{gou1,gou2,hai}.

We write $\lambda=(\lambda_1,\ldots,\lambda_\ell)\vdash n$ for a
partition of $n$ and denote its length by $\ell(\lambda)=\ell$.
Let $\mathfrak{S}_\lambda\cong
\mathfrak{S}_{\lambda_1}\times\cdots\times
\mathfrak{S}_{\lambda_\ell}$ be the standard Young subgroup, and set
$\eta^\lambda=\operatorname{Ind}_{\mathfrak{S}_\lambda}^{\mathfrak{S}_n}
\mathbf{1}$. The Young permutation characters $\eta^\lambda$ form a
basis of the space of real-valued class functions. Their dual basis
$\phi^\lambda$ is characterized by
\begin{equation*}\label{eq2}
 \langle\eta^\mu,\phi^\lambda\rangle=\delta_{\lambda\mu},
 \qquad
 \langle f,g\rangle=\frac{1}{n!}
 \sum_{\sigma\in\mathfrak{S}_n}f(\sigma)g(\sigma^{-1}).
\end{equation*}
Under the Frobenius characteristic map, $\eta^\lambda$ and
$\phi^\lambda$ correspond to the complete homogeneous symmetric
function $h_\lambda$ and the monomial symmetric function $m_\lambda$,
respectively \cite{ste1}. The functions $\phi^\lambda$ are virtual
characters, and the associated expressions $\phi^\lambda[A]$ are called
\emph{monomial immanants}. Since virtual characters may take negative
values, their nonnegativity on a matrix class requires cancellation
among the terms of the defining sum.

A real $n\times n$ matrix is \emph{totally nonnegative} if all its minors
are nonnegative. Following the notation used here and in
\cite{ste1}, we denote this class by $\operatorname{TP}$.
More generally, $\operatorname{TP}_r(n)$ consists of matrices whose
minors of order at most $r$ are nonnegative; we abbreviate this to
$\operatorname{TP}_r$ when the matrix size is understood. The indices
of every minor are taken in their natural order. Thus
$\operatorname{TP}_1(n)$ is the class of entrywise nonnegative matrices,
and $\operatorname{TP}_n(n)=\operatorname{TP}$. Throughout the paper,
the notation $\operatorname{TP}_r$ allows zero minors.

\subsection{Stembridge's question and previous progress}\label{subsec:history}
In his influential study of immanant inequalities, Stembridge
proposed the following assertion.

\begin{conjecture}\cite[Conjecture~2.1]{ste1}\label{con1.1}
For every partition $\lambda\vdash n$ and every totally nonnegative
$n\times n$ matrix $A$, one has $\phi^\lambda[A]\geq0$.
\end{conjecture}

The role of monomial immanants is explained in part by the Kostka
expansion: every irreducible character is a nonnegative integer linear
combination of the $\phi^\lambda$ \cite[Section~1]{ste1}.
Consequently, Conjecture~\ref{con1.1} strengthens the known
nonnegativity of irreducible character immanants on totally nonnegative
matrices. Stembridge then asked whether the matrix hypothesis could
be weakened according to the length of the indexing partition.

\begin{problem}\cite[Question~2.9]{ste1}\label{prob1.2}
Let $\lambda\vdash n$, let $r$ be a positive integer, and let
$A\in\operatorname{TP}_r(n)$. If $\ell(\lambda)\leq r$, must
$\phi^\lambda[A]\geq0$?
\end{problem}

Taking $r=n$ recovers Conjecture~\ref{con1.1}. For smaller $r$,
Problem~\ref{prob1.2} asks whether minors only up to the partition
length already control the sign of the immanant. This distinction is
essential: a matrix in $\operatorname{TP}_r(n)$ may have negative
minors of larger order.

Stembridge established Conjecture~\ref{con1.1} for
$\lambda=(2,1^{n-2})$ and for rectangular partitions
\cite[Theorems~2.7 and~2.8]{ste1}. He also explained that these cases
satisfy the stronger assertion in Problem~\ref{prob1.2}, observed the
case $r=1$, and verified the case $r=2$, $n\leq5$ by direct computation
\cite[discussion following Question~2.9]{ste1}.

Subsequent work developed algebraic and combinatorial approaches to
these positivity questions. Haiman \cite[Conjecture~2.1]{hai}
formulated a stronger conjecture asserting coefficientwise
nonnegativity for monomial traces evaluated on suitably normalized
Kazhdan--Lusztig basis elements of the Hecke algebra. Its specialization
at $q=1$ implies Conjecture~\ref{con1.1}.
Clearman, Shelton, and Skandera \cite[Theorems~5.6 and~5.7]{cle1}
proved Conjecture~\ref{con1.1} for all partitions with at most two
columns, equivalently $\lambda_1\leq2$, by expanding the corresponding
monomial immanants as nonnegative integer linear combinations of
Temperley--Lieb immanants and giving interpretations by path tableaux.
Their hook-shaped path-tableau formula
\cite[Theorem~6.2]{cle1} concerns irreducible character immanants;
the monomial immanants indexed by arbitrary hooks require a separate
argument.

More recently, Lesnevich \cite{les1} studied the related
Schur-positivity problem for monomial immanants of Jacobi--Trudi
matrices. He expressed hook-shape immanant characters as nonnegative
integer combinations of Stanley--Stembridge characters. Here the hook
shape indexes a Schur coefficient in the symmetric-function expansion,
so this result addresses a different aspect of immanant positivity
from the hook index of $\phi^{(a,1^k)}$ considered below.
Skandera \cite[Proposition~4.10]{ska1} proved that, for each fixed
length $\ell$, the sum
\[
 \sum_{\substack{\lambda\vdash n\\\ell(\lambda)=\ell}}
 \phi^\lambda[A]
\]
is nonnegative on totally nonnegative matrices, and obtained a
weighted path-tableau interpretation. This establishes positivity
after summing over a fixed length, while
Problem~\ref{prob1.2} concerns each individual immanant under a
hypothesis involving only minors of bounded order.

\subsection{Main result and contributions}\label{subsec:main}

The main result of this paper gives a complete affirmative answer to
Problem~\ref{prob1.2} within the family of hook partitions.

\begin{theorem}\label{thm1.1}
Let $a\geq1$ and $k\geq0$ be integers, let $n=a+k$, and let $A$ be an
$n\times n$ real matrix. If $A\in\operatorname{TP}_{k+1}(n)$, then
\[
 \phi^{(a,1^k)}[A]\geq0.
\]
\end{theorem}

Since $\ell((a,1^k))=k+1$ and
$\operatorname{TP}_r(n)\subseteq\operatorname{TP}_{k+1}(n)$ whenever
$r\geq k+1$, Theorem~\ref{thm1.1} proves precisely the implication
asked for in Problem~\ref{prob1.2} for every hook. In particular, it
also settles the hook case of Conjecture~\ref{con1.1}. The theorem
imposes no bound on the first part or the leg length, and requires
nonnegativity only through the minor order specified in Stembridge's
question. For example, it gives
$\phi^{(n-1,1)}[A]\geq0$ on $\operatorname{TP}_2(n)$ for every
$n\geq2$, and $\phi^{(n-2,1^2)}[A]\geq0$ on
$\operatorname{TP}_3(n)$ for every $n\geq3$.

A further contribution is a refinement of the desired inequality.
For each marked vertex $v$, Definition~\ref{def1} introduces a signed
sum $\Psi_v^{a,k}(A)$ over permutations whose cycle containing $v$ has
length at least $a$. The proof of Theorem~\ref{thm1.1} establishes
$\Psi_v^{a,k}(A)\geq0$ for every $v$ separately.
Identity~(\ref{eq5}) then expresses a positive multiple of the hook
monomial immanant as the sum of these nonnegative quantities. Thus the
argument explains positivity at the level of individual marked cycle
sums, before summing over the marked vertex.

The mechanism behind this refinement is the nonnegative recurrence
(\ref{eq17}). Its generating-function form is
Lemma~\ref{lem9}, and the coefficientwise nonnegativity needed to use
it is supplied by Lemma~\ref{lem10}. Together these results provide a
uniform proof for all hooks, including matrices with zero entries or
vanishing minors.

\subsection{Basic ideas of the proof}\label{subsec:strategy}

The first step separates the character-theoretic input from the matrix
inequality. Stembridge's hook character formula, recalled in
Lemma~\ref{lem2}, weights a permutation by the number of vertices lying
in cycles of length at least $a$, up to an explicit sign and a positive
normalizing factor. Marking such a vertex gives the decomposition
(\ref{eq5}). The resulting marked sums still contain alternating
signs, so the decomposition alone does not prove nonnegativity.

To organize these cancellations, we introduce one variable for each
unmarked vertex. A formal resolvent records closed walks through the
marked vertex, while a determinant records principal minors on the
remaining vertices. Extracting a squarefree coefficient forces the
walk to use distinct vertices and keeps its vertices disjoint from
those used by the determinant. Lemma~\ref{lem7} identifies the
resulting coefficient with $\Psi_v^{a,k}(A)$. This encoding turns the
cycle restriction into a coefficient-extraction problem compatible
with block matrix identities.

For $k\geq1$, the next step is to eliminate an endpoint $w$ of the
ordered index set, chosen different from the marked vertex $v$. If $s=a_{ww}>0$ and $C$
is the Schur complement at $w$, Lemma~\ref{lem5} gives
\[
 A\in\operatorname{TP}_{k+1}
 \quad\Longrightarrow\quad C\in\operatorname{TP}_{k}.
\]
Choosing an endpoint ensures that the minors of $C$ are ratios of
minors of $A$ with no additional sign. The block identities in
Lemma~\ref{lem8} then yield the generating-function recurrence of
Lemma~\ref{lem9}. One term replaces $A$ by $C$ and lowers $k$ by one;
the other replaces $A$ by smaller principal submatrices and transfers
the marked vertex from $v$ to $w$.

The remaining difficulty is the sign of the transfer coefficients.
Their initial expression is a quotient of formal series, from which
nonnegativity is not apparent. Lemma~\ref{lem10} supplies a
factorization in terms of nonnegative matrix blocks and formal
resolvents with nonnegative coefficients. Consequently, every
coefficient in the extracted recurrence (\ref{eq17}) is nonnegative.
The Schur-complement term lowers both the leg length and the required
minor order by one, whereas each principal-submatrix term preserves
the minor condition and reduces the matrix size. Strong induction
therefore proves the nonnegativity of all marked cycle sums
simultaneously. The case $k=0$ is a sum of nonnegative cycle weights;
when an endpoint pivot vanishes, Lemma~\ref{lem5} gives a zero row or
column, so the relevant sum vanishes. These observations complete the
argument without a strict-positivity assumption.

\medskip
\noindent{\bf Outline of the paper.}
Section~\ref{sec2} develops the marked cycle generating functions and
proves Theorem~\ref{thm1.1}. After recalling the character identity in
Lemma~\ref{lem2} and the Schur-complement facts in
Lemmas~\ref{lem5} and~\ref{lem6}, we establish the coefficient formula
in Lemma~\ref{lem7}, derive the elimination identities in
Lemmas~\ref{lem8} and~\ref{lem9}, and prove the nonnegativity of the
transfer series in Lemma~\ref{lem10}. We then extract recurrence
(\ref{eq17}) and carry out the induction.
 
\section{Marked cycle sums and a nonnegative recurrence}\label{sec2}
We prove Theorem~\ref{thm1.1} by first separating the contribution of
each marked vertex and then deriving a recurrence whose coefficients
are nonnegative. The notation is introduced in groups as it becomes
needed in the argument.

\subsection{Basic notation and marked cycle sums}\label{subsec:marked-cycles}

Throughout this section, $V$ is a finite nonempty totally ordered set,
$A=(a_{ij})_{i,j\in V}$ is a real matrix, and
$[n]=\{1,2,\ldots,n\}$. We begin with three conventions for the
character formula and the marked cycle sums.

\begin{enumerate}
\item \textbf{Permutations and cycles.}
The symmetric group on $V$ is denoted by $\mathfrak{S}(V)$.
Its class functions, cycle types, and signs are identified with those
on $\mathfrak{S}_{|V|}$ through the order-preserving bijection
$V\to[|V|]$. For $\sigma\in\mathfrak{S}(V)$, let
$\ell_v(\sigma)$ be the length of the cycle containing $v$, with a
fixed point regarded as a $1$-cycle.

\item \textbf{Permutation weights.}
The weight of $\sigma\in\mathfrak{S}(V)$ is
$\operatorname{wt}_A(\sigma)=\prod_{i\in V}a_{i\sigma(i)}$.

\item \textbf{Hook normalization.}
For $a\geq1$ and $k\geq0$, let $\kappa_{a,k}$ be the multiplicity of
the part $a$ in $(a,1^k)$; explicitly,
\[
 \kappa_{a,k}=\begin{cases}1,&a>1,\\ k+1,&a=1.\end{cases}
\]
\end{enumerate}

The following formula of Stembridge \cite{ste1} expresses a hook
monomial character in terms of the vertices lying in sufficiently long
cycles, and provides the character-theoretic starting point of the proof.

\begin{lemma}[\cite{ste1}]\label{lem2}
Let $a\geq1$, $k\geq0$, $a+k=n$, $\sigma\in\mathfrak{S}_n$, and let $m_q(\sigma)$ denote the number of $q$-cycles of $\sigma$. Then
\begin{equation}\label{eq3}
 \kappa_{a,k}\phi^{(a,1^k)}(\sigma)
 =(-1)^{a-1}\operatorname{sgn}(\sigma)\sum_{q\geq a}q\,m_q(\sigma).
\end{equation}
\end{lemma}

Since each cycle of length $q$ contains exactly $q$ possible marked
vertices, Lemma~\ref{lem2} suggests isolating the contribution associated
with one fixed vertex.

\begin{definition}\label{def1}
Let $|V|=a+k$, $a\geq1$, $k\geq0$, $v\in V$. Define
\begin{equation}\label{eq4}
 \Psi_v^{a,k}(A)=(-1)^{a-1}
 \sum_{\substack{\sigma\in\mathfrak{S}(V)\\\ell_v(\sigma)\geq a}}
 \operatorname{sgn}(\sigma)\operatorname{wt}_A(\sigma),
\end{equation}
called the \emph{cycle sum with marked vertex} \(v\).
\end{definition}

\subsection{Endpoint Schur complements}\label{subsec:endpoint-schur}

We next introduce the matrix notation needed for the elimination step.
The order on the index set matters because it determines the signs of
all minors.

\begin{enumerate}
\item \textbf{Submatrices and minors.}
For $I,J\subseteq V$, let $A(I\mid J)$ denote the submatrix with rows
in $I$ and columns in $J$, both in increasing order, and write
$A[I]=A(I\mid I)$. When $|I|=|J|$, set
\[
 \Delta_{I\mid J}(A)=\operatorname{det}A(I\mid J),\qquad
 \Delta_I(A)=\operatorname{det}A[I].
\]
We use the convention
$\operatorname{det}A[\varnothing]=\Delta_\varnothing(A)=1$.
For individual indices, row and column blocks are abbreviated as
$A_{vW}=A(\{v\}\mid W)$ and $A_{Wv}=A(W\mid\{v\})$;
the same convention applies to other blocks.

\item \textbf{Minor conditions and endpoints.}
For $r\geq1$, define
\[
 \operatorname{TP}_r(V)=
 \{A:\Delta_{I\mid J}(A)\geq0
 \text{ whenever }I,J\subseteq V,\ |I|=|J|\leq r\}.
\]
If $r>|V|$, only existing minors are considered. An \emph{endpoint}
of $V$ is its minimum or maximum element. All subsets of $V$ inherit
its order.
\end{enumerate}

The induction will eliminate an endpoint of the index set; the following
lemma explains how this operation lowers the required order of total
nonnegativity and handles a zero pivot. Its determinant identity is a
standard consequence of the Schur determinant formula \cite{hor1}.

\begin{lemma}[\cite{hor1}]\label{lem5}
Let $A\in\operatorname{TP}_r(V)$, $r\geq2$, and let $w$ be an endpoint of $V$.
\begin{enumerate}[label=(\alph*)]
\item If $a_{ww}=0$, then either the \(w\)-th row or the \(w\)-th column of \(A\) is identically zero. \label{lem5:ita}

\item If $s=a_{ww}>0$, let $U=V\setminus\{w\}$, and define the {\emph Schur complement $C$ with respect to the pivot $s$} to be
\[
 C=A[U]-\frac1s A(U\mid\{w\})A(\{w\}\mid U),
 \qquad c_{ij}=a_{ij}-\frac{a_{iw}a_{wj}}s, \qquad i,j\in U.
\]
For any $I,J\subseteq U$ with $|I|=|J|$, we have
\begin{equation}\label{eq6}
 \operatorname{det} C(I\mid J)
 =\frac{\operatorname{det} A(I\cup\{w\}\mid J\cup\{w\})}{s}.
\end{equation}
In particular, $C\in\operatorname{TP}_{r-1}(U)$.    \label{lem5:itb}
\end{enumerate}
\end{lemma}

\begin{proof}
\ref{lem5:ita}
Assume $a_{ww}=0$. Since $w$ is an endpoint, for any $i,j\in V\setminus\{w\}$, the nonnegativity of second-order minors gives
\[
 0\leq\operatorname{det} A(\{i,w\}\mid\{j,w\})=-a_{iw}a_{wj}.
\]
Since \(A\) is entrywise nonnegative, \(a_{iw}a_{wj}=0\). If there exists \(i\neq w\) such that \(a_{iw}>0\), then \(a_{wj}=0\) for all \(j\neq w\), and hence the \(w\)-th row is identically zero. If no such \(i\) exists, then the \(w\)-th column is identically zero.

\ref{lem5:itb} Assume $s=a_{ww}>0$. For any $I,J\subseteq U$, $|I|=|J|$, simultaneously move the row and column of the submatrix $A(I\cup\{w\}\mid J\cup\{w\})$ corresponding to $w$ to the first position. When $w$ is the minimum element, no rearrangement is needed. When $w$ is the maximum element, the rearrangements of rows and columns produce signs $(-1)^{|I|}$ and $(-1)^{|J|}$, respectively, whose product is $1$. Therefore, by the Schur determinant formula, we obtain
\[
 \begin{aligned}
 \operatorname{det} A(I\cup\{w\}\mid J\cup\{w\})
 &=\operatorname{det}\begin{pmatrix}
 s&A(\{w\}\mid J)\\
 A(I\mid\{w\})&A(I\mid J)
 \end{pmatrix}\\
 &=s\operatorname{det}\left(A(I\mid J)-\frac1s A(I\mid\{w\})A(\{w\}\mid J)\right)\\
 &=s\operatorname{det} C(I\mid J).
 \end{aligned}
\]
This yields \eqref{eq6}. When $|I|=|J|\leq r-1$, the corresponding minor of \(A\) on the right-hand side has order at most \(r\), so $\operatorname{det} C(I\mid J)\geq0$. Hence $C\in\operatorname{TP}_{r-1}(U)$.
\end{proof}

To carry out the same elimination at the level of generating functions,
we also recall the block determinant and rank-one update identities in
their formal power series form.

\begin{lemma}\cite{hor1}\label{lem6}
Let $T$ be a square matrix whose entries lie in the ring of commutative formal power series over the reals, and whose constant term matrix is invertible. Let $b$ be a scalar, $\rho$ a row vector, and $\gamma$ a column vector, with dimensions compatible with $T$.
Suppose $\xi$ and $\zeta$ are any two compatible column vectors. Then
\begin{equation}\label{eq9}
 \begin{aligned}
 \operatorname{det}\begin{pmatrix}b&\rho\\\gamma&T\end{pmatrix}
 &=\operatorname{det} T\,(b-\rho T^{-1}\gamma),\\
 \operatorname{det}(T+\xi\zeta^{\mathsf T})
 &=\operatorname{det} T\,(1+\zeta^{\mathsf T}T^{-1}\xi).
 \end{aligned}
\end{equation}
\end{lemma}

\subsection{Generating functions and coefficient extraction}\label{subsec:cycle-series}

We now encode the marked cycle sums by formal series. The following
conventions separate the vertex bookkeeping from the matrix operations
and introduce the generating functions used in Lemma~\ref{lem7}.

\begin{enumerate}
\item \textbf{Vertex variables and squarefree monomials.}
For a finite ordered set $W$, introduce independent commuting variables
$z_i$, $i\in W$, and set
\[
 Z_W=\operatorname{diag}(z_i:i\in W),\qquad
 z_S=\prod_{i\in S}z_i\quad(S\subseteq W),\qquad z_\varnothing=1.
\]
The monomial $z_S$ is squarefree: each vertex in $S$ occurs once.

\item \textbf{Coefficient extraction.}
For a formal series $F$, the notation $[z_S]F$ selects exponent $1$
for variables indexed by $S$ and exponent $0$ for the other vertex
variables, while retaining the parameters $u,t,x$.
Similarly, $[u^pt^jz_S]F$ extracts the indicated parameter and vertex
exponents, whereas $[u^pt^j]F$ and $[x^d]F$ retain all vertex variables.
To extract only one vertex variable, write
$[z_w^m]_{z_w}F=F_m$ when $F=\sum_{m\geq0}F_mz_w^m$ and each $F_m$
is independent of $z_w$. In particular, for $R=W\setminus\{w\}$,
\[
 [u^pt^jz_R]\bigl([z_w]_{z_w}F\bigr)=[u^pt^jz_W]F.
\]
All series have nonnegative exponents, so squarefree coefficients
satisfy the product rule
\begin{equation}\label{eq7}
 [z_S](FG)=\sum_{T\subseteq S}([z_T]F)([z_{S\setminus T}]G).
\end{equation}

\item \textbf{Formal inverses and a divided difference.}
All matrix inverses below are formal inverses; in particular,
\[
 (I-xMZ_W)^{-1}=\sum_{d\geq0}x^d(MZ_W)^d.
\]
When $W=\varnothing$, the corresponding determinant equals $1$ and
products of empty row and column blocks equal $0$.
For $q(x)=\sum_{d\geq0}q_dx^d$, we use the notation
\begin{equation}\label{eq8}
 \frac{uq(u)+tq(-t)}{u+t}
 :=\sum_{d\geq0}q_d\sum_{h=0}^d(-1)^hu^{d-h}t^h.
\end{equation}
The coefficients $q_d$ may themselves involve vertex variables.
Multiplication by $u+t$ recovers $uq(u)+tq(-t)$, so
(\ref{eq8}) does not require $u+t$ to be invertible.

\item \textbf{Closed-walk and marked cycle series.}
Fix $v\in V$, and put $W=V\setminus\{v\}$ and $E=A[W]$.
Define
\begin{equation}\label{eq10}
 \begin{aligned}
 q_{A,v}(x)&=a_{vv}+xA_{vW}Z_W(I-xEZ_W)^{-1}A_{Wv},\\
 K_{A,v}(u,t)&=\operatorname{det}(I+tEZ_W)
       \frac{uq_{A,v}(u)+tq_{A,v}(-t)}{u+t}.
 \end{aligned}
\end{equation}
Here $q_{A,v}$ records closed walks starting at $v$ whose internal
vertices lie in $W$, while the determinant factor in $K_{A,v}$ records
principal minors. Squarefree coefficient extraction will ensure that
the vertices used by these two factors are disjoint.

\item \textbf{Cycle weights on a prescribed vertex set.}
For $T\subseteq W$, let $\mathcal{O}(T)$ be the set of linear orderings
of $T$, and define
\[
 \Gamma_v(T)=
 \begin{cases}
 a_{vv},&T=\varnothing,\\[1mm]
 \displaystyle\sum_{(t_1,\ldots,t_d)\in\mathcal{O}(T)}
 a_{vt_1}\left(\prod_{h=1}^{d-1}a_{t_ht_{h+1}}\right)a_{t_dv},
 &d=|T|\geq1.
 \end{cases}
\]
For nonempty $T$, each cycle
$$v\to t_1\to\cdots\to t_d\to v$$ is counted once; the empty set
corresponds to the fixed point $v$. Empty products are understood to
be $1$.
\end{enumerate}

The next lemma identifies the marked cycle sums with squarefree
coefficients of $K_{A,v}$ and records the locality property needed to
apply the resulting recurrence to smaller principal submatrices.

\begin{lemma}\label{lem7}
Let $S\subseteq W$, and let $p,j$ be nonnegative integers.
\begin{enumerate}[label=(\alph*)]
\item If $|S|=p+j$, then
\begin{equation}\label{eq11}
 [u^pt^jz_S]K_{A,v}
 =\sum_{\substack{I\subseteq S\\|I|\leq j}}
   (-1)^{j-|I|}\operatorname{det} A[I]\,\Gamma_v(S\setminus I).
\end{equation}
In particular, when $|V|=a+k$,
\begin{equation}\label{eq12}
 \Psi_v^{a,k}(A)=[u^{a-1}t^kz_W]K_{A,v}.
\end{equation} \label{lem7:ita}

\item Moreover, $[u^pt^j]K_{A,v}$ is homogeneous of degree $p+j$ in the vertex variables, and the coefficients in \eqref{eq11} depend only on the principal submatrix $A[\{v\}\cup S]$. \label{lem7:itb}
\end{enumerate}
\end{lemma}

\begin{proof}
\ref{lem7:ita} By \eqref{eq10}, for $d\geq1$, we obtain
\[
 \begin{aligned}
 {}[x^d]q_{A,v}(x)
 &=A_{vW}Z_W(EZ_W)^{d-1}A_{Wv}\\
 &=\sum_{i_1,\ldots,i_d\in W}
 a_{vi_1}a_{i_1i_2}\cdots a_{i_{d-1}i_d}a_{i_dv}
 z_{i_1}\cdots z_{i_d}.
 \end{aligned}
\]
If $|T|=d$, then when extracting the coefficient of the squarefree monomial $z_T$, the vertices in the sequence $i_1,\ldots,i_d$ must be pairwise distinct and must constitute exactly $T$. Hence these sequences are precisely the elements of $\mathcal{O}(T)$, and the corresponding closed walks are all cycles containing $v$. Thus $[x^{|T|}z_T]q_{A,v}(x)=\Gamma_v(T)$, which also holds for $T=\varnothing$. By the principal minor expansion of the determinant,
\[
 \operatorname{det}(I+tEZ_W)=\sum_{I\subseteq W}t^{|I|}z_I\operatorname{det} A[I].
\]
By \eqref{eq7}, if the determinant part uses the vertex set $I\subseteq S$, then the walk part uses $S\setminus I$. Let $d=|S\setminus I|$. By \eqref{eq8}, the latter contributes
\[
 \Gamma_v(S\setminus I)\sum_{h=0}^{d}(-1)^h u^{d-h}t^h.
\]
After multiplying by $t^{|I|}\operatorname{det} A[I]$, the coefficient of $u^pt^j$ corresponds to $h=j-|I|$. Since $|S|=p+j$ and $p\geq0$, the condition $0\leq h\leq d$ is equivalent to $|I|\leq j$, which yields \eqref{eq11}.
Take $S=W$, $p=a-1$, $j=k$. For each $I$, merge the cycle of length $n-|I|\geq a$ in $\Gamma_v(W\setminus I)$ with the permutation $\pi$ in $\operatorname{det} A[I]$ to obtain a permutation $\sigma$, where $n=a+k$.
Let $d=n-1-|I|$. Then $\operatorname{sgn}(\sigma)=(-1)^d\operatorname{sgn}(\pi)$, and so
\[
 (-1)^{a-1}\operatorname{sgn}(\sigma)
 =(-1)^{a-1+d}\operatorname{sgn}(\pi)
 =(-1)^{k-|I|}\operatorname{sgn}(\pi),
\]
where $a-1+d$ and $k-|I|$ have the same parity. Every permutation satisfying $\ell_v(\sigma)\geq a$ has a unique decomposition as above, and its sign agrees with \eqref{eq11}; hence we obtain \eqref{eq12}. 

\ref{lem7:itb}  From the above expansion, the total degree in $u,t$ equals the total degree in the vertex variables, so $[u^pt^j]K_{A,v}$ is homogeneous of degree $p+j$ in the vertex variables. Equation \eqref{eq11} involves only matrix entries within $\{v\}\cup S$, so the corresponding coefficients depend only on $A[\{v\}\cup S]$.
\end{proof}

\subsection{Endpoint elimination and transfer coefficients}\label{subsec:elimination}

To derive the recurrence, we compare the generating functions before
and after eliminating a second vertex $w\neq v$. We collect the
notation for this comparison here; the algebraic identities require
only a nonzero pivot, and the endpoint condition will enter when we
prove nonnegativity.

\begin{enumerate}
\item \textbf{The pivot and reduced matrices.}
Let $s=a_{ww}\neq0$, and let $C$ be the Schur complement at $w$,
with entries $c_{ij}=a_{ij}-a_{iw}a_{wj}/s$ for
$i,j\in V\setminus\{w\}$. Set
\[
 R=V\setminus\{v,w\},\qquad
 F=A[R],\qquad D=C[R],\qquad
 N=A[V\setminus\{v\}],\qquad Z=Z_R.
\]

\item \textbf{Four auxiliary walk series.}
Writing $M_F(x)=(I-xFZ)^{-1}$, define
\[
 \begin{aligned}
 \alpha(x)&=a_{vv}+xA_{vR}ZM_F(x)A_{Rv},\\
 \beta(x)&=a_{vw}+xA_{vR}ZM_F(x)A_{Rw},\\
 \gamma(x)&=a_{wv}+xA_{wR}ZM_F(x)A_{Rv},\\
 \delta(x)&=s+xA_{wR}ZM_F(x)A_{Rw}.
 \end{aligned}
\]
These series record walks between $v$ and $w$ with internal vertices
in $R$.

\item \textbf{The transfer series.}
Since $\delta(0)=s\neq0$, the series $\delta$ is invertible, and we set
\[
 \Omega_{A;v,w}(u)=\frac{\beta(u)\gamma(u)}{\delta(u)}.
\]
This series will supply the coefficients of the term that transfers
the marked vertex from $v$ to $w$.
\end{enumerate}

With this notation, the block identities below describe how the
determinant and closed-walk series change when the vertex $w$ is
eliminated.

\begin{lemma}\label{lem8}
Assume $v,w\in V$ with $v\neq w$, $s=a_{ww}\neq0$,
and retain the notation $C,R,F,D,N,\allowbreak\alpha,\allowbreak\beta,\allowbreak\gamma,\allowbreak\delta$ above. Let $\mathcal{D}_F(t)=\operatorname{det}(I+tFZ)$, $\mathcal{D}_D(t)=\operatorname{det}(I+tDZ)$. Then
\begin{equation}\label{eq13}
 \begin{aligned}
 \operatorname{det}(I+tNZ_{V\setminus\{v\}})
   &=\mathcal{D}_F(t)(1+tz_w\delta(-t)),\\
 q_{A,v}(x)&=\alpha(x)+
       \frac{xz_w\beta(x)\gamma(x)}{1-xz_w\delta(x)},\\
 \mathcal{D}_D(t)&=\mathcal{D}_F(t)\frac{\delta(-t)}s,\qquad
 q_{N,w}(x)=\delta(x),\\
 q_{C,v}(x)&=\alpha(x)-\frac{\beta(x)\gamma(x)}{\delta(x)}.
 \end{aligned}
\end{equation}
\end{lemma}
\begin{proof}
Simultaneously reorder the row and column indices of \(N\) into the order \(w,R\), and perform the same reordering on the corresponding diagonal variable matrix. Under this arrangement, we obtain
\[
 I+tNZ_{V\setminus\{v\}}
 =\begin{pmatrix}
 1+tsz_w&tA_{wR}Z\\
 tA_{Rw}z_w&I+tFZ
 \end{pmatrix}.
\]
The above simultaneous reordering is equivalent to a permutation similarity transformation, so it does not change the determinant. The following block computations are all carried out in this order, and do not concern the total nonnegativity of the reordered matrix. By \eqref{eq9}, the Schur complement of the lower-right block is
\[
 1+tsz_w-t^2z_wA_{wR}Z(I+tFZ)^{-1}A_{Rw}
 =1+tz_w\delta(-t).
\]
This gives the first identity.

Let \(y=(y_w,y_R^{\mathsf T})^{\mathsf T}\) satisfy
\[
 (I-xNZ_{V\setminus\{v\}})y=A_{V\setminus\{v\},v}.
\]
The constant term of the coefficient matrix is the identity matrix, so \(y\) exists uniquely. Partitioning the equations gives
\begin{align}
 (1-xs z_w)y_w-xA_{wR}Zy_R&=a_{wv}, \label{eq13.1}\\
 -xz_wA_{Rw}y_w+(I-xFZ)y_R&=A_{Rv}. \label{eq13.2}
\end{align}
From \eqref{eq13.2}, we get
\[
 y_R=M_F(x)A_{Rv}+xz_wM_F(x)A_{Rw}y_w.
\]
Substituting into \eqref{eq13.1}, we obtain
\[
 \bigl(1-xs z_w-x^2z_wA_{wR}ZM_F(x)A_{Rw}\bigr)y_w
 =a_{wv}+xA_{wR}ZM_F(x)A_{Rv}.
\]
By the definitions of \(\delta\) and \(\gamma\),
\[
 (1-xz_w\delta(x))y_w=\gamma(x),\qquad
 y_w=\frac{\gamma(x)}{1-xz_w\delta(x)}.
\]
Substituting \(y_R\) and \(y_w\) into the definition of \(q_{A,v}\), we obtain
\[
 q_{A,v}(x)=\alpha(x)+xz_w\beta(x)y_w
 =\alpha(x)+\frac{xz_w\beta(x)\gamma(x)}{1-xz_w\delta(x)}.
\]
Since \(D=F-A_{Rw}A_{wR}/s\), we have
\[
 I+tDZ=(I+tFZ)-\frac ts A_{Rw}A_{wR}Z.
\]
By the rank-one determinant identity in \eqref{eq9}, we obtain
\[
 \mathcal{D}_D(t)=\mathcal{D}_F(t)
 \left(1-\frac ts A_{wR}Z(I+tFZ)^{-1}A_{Rw}\right)
 =\mathcal{D}_F(t)\frac{\delta(-t)}s.
\]
Also \(N[R]=F\), so from \eqref{eq10} we have \(q_{N,w}(x)=\delta(x)\).

Consider the system of equations for the scalar \(b\) and the column vector \(y_R\):
\begin{align}
 s b+A_{wR}Zy_R&=a_{wv}, \label{eq9.1}\\
 -xA_{Rw}b+(I-xFZ)y_R&=-xA_{Rv}. \label{eq9.2}
\end{align}
Its constant term matrix with respect to all variables is \(\operatorname{diag}(s,I)\), so it has a unique formal power series solution. Consider the scalar
\[
 a_{vv}-a_{vw}b-A_{vR}Zy_R.
\]
Eliminating \(y_R\) using \eqref{eq9.2}, we get
\[
 y_R=-xM_F(x)A_{Rv}+xM_F(x)A_{Rw}b.
\]
Substituting into \eqref{eq9.1}, we have \(\delta(x)b=\gamma(x)\), so the above scalar equals
\[
 \alpha(x)-\frac{\beta(x)\gamma(x)}{\delta(x)}.
\]
On the other hand, eliminating \(b\) from \eqref{eq9.1}, we get
\[
 b=\frac{a_{wv}-A_{wR}Zy_R}{s}.
\]
Substituting into \eqref{eq9.2} and using the definition of the Schur complement, we obtain
\[
 (I-xDZ)y_R=-xC_{Rv},\qquad
 y_R=-x(I-xDZ)^{-1}C_{Rv}.
\]
Then the above scalar is
\[
 c_{vv}-C_{vR}Zy_R
 =c_{vv}+xC_{vR}Z(I-xDZ)^{-1}C_{Rv}
 =q_{C,v}(x).
\]
By uniqueness of the solution, the two expressions are equal.
\end{proof}

\begingroup
\emergencystretch=1em
Extracting the coefficient of $z_w$ from the identities in
Lemma~\ref{lem8} separates the Schur-complement contribution from the
term that transfers the marked vertex from $v$ to $w$.\par
\endgroup

\begin{lemma}\label{lem9}
Retaining the assumptions and notation of Lemma \ref{lem8}, we have
\begin{equation}\label{eq14}
 [z_w]_{z_w}K_{A,v}(u,t)
 =tsK_{C,v}(u,t)+u\Omega_{A;v,w}(u)K_{N,w}(u,t).
\end{equation}
Both sides retain $u,t$ and the vertex variables on $R$.
\end{lemma}
\begin{proof}
Write $\alpha_u=\alpha(u)$, $\alpha_-=\alpha(-t)$, and adopt the same convention for $\beta,\gamma,\delta$.
By \eqref{eq13}, the constant term of $q_{A,v}(x)$ with respect to \(z_w\) is $\alpha(x)$, and the coefficient of the linear term is $x\beta(x)\gamma(x)$. Therefore
\[
 [z_w]_{z_w}q_{A,v}(u)=u\beta_u\gamma_u,\qquad
 [z_w]_{z_w}q_{A,v}(-t)=-t\beta_-\gamma_-.
\]
Since the determinant factor is $\mathcal{D}_F(t)(1+tz_w\delta_-)$, we have
\[
 \begin{aligned}
 (u+t)[z_w]_{z_w}K_{A,v}
 =\mathcal{D}_F(t)\{t\delta_-(u\alpha_u+t\alpha_-)
       +u^2\beta_u\gamma_u-t^2\beta_-\gamma_-\}.
 \end{aligned}
\]
Then using $\mathcal{D}_D(t)=\mathcal{D}_F(t)\delta_-/s$ and
$q_{C,v}=\alpha-\beta\gamma/\delta$, we obtain
\[
 \begin{aligned}
 (u+t)tsK_{C,v}
 =\mathcal{D}_F(t)\biggl\{&t\delta_-(u\alpha_u+t\alpha_-)
 -tu\delta_-\frac{\beta_u\gamma_u}{\delta_u}
 -t^2\beta_-\gamma_-\biggr\}.
 \end{aligned}
\]
Subtracting the two identities, we obtain
\[
 \begin{aligned}
 &(u+t)\bigl([z_w]_{z_w}K_{A,v}-tsK_{C,v}\bigr)\\
 &\quad=\mathcal{D}_F(t)\left(u^2\beta_u\gamma_u
       +tu\delta_-\frac{\beta_u\gamma_u}{\delta_u}\right)\\
 &\quad=u\frac{\beta_u\gamma_u}{\delta_u}\,
       \mathcal{D}_F(t)(u\delta_u+t\delta_-).
 \end{aligned}
\]
Since \(q_{N,w}=\delta\), the right-hand side equals $(u+t)u\Omega_{A;v,w}K_{N,w}$. The above identity holds in the ring $\mathbb{R}[[u,t,z_i\ (i\in R)]]$ of formal power series with real coefficients. Since this ring has no zero divisors, we may cancel the nonzero factor \(u+t\), yielding \eqref{eq14}.
\end{proof}

To use the recurrence in Lemma~\ref{lem9} for a positivity argument,
we must control the coefficients of the transfer series; the following
factorization makes their nonnegativity explicit under the endpoint
and $\operatorname{TP}_2$ assumptions.

\begin{lemma}\label{lem10}
Let $A\in\operatorname{TP}_2(V)$, let \(w\) be an endpoint, $v\neq w$, $s=a_{ww}>0$.
Then all coefficients of $\Omega_{A;v,w}(u)$ with respect to \(u\) and \(z_i\) $(i\in R)$ are nonnegative.
Specifically,
\begin{equation}\label{eq15}
 \begin{aligned}
 \Omega_{A;v,w}(u)
 ={}&\left(\frac{a_{vw}}s+
       \frac us C_{vR}Z(I-uDZ)^{-1}A_{Rw}\right)\\
 &\quad\cdot\left(a_{wv}+uA_{wR}Z(I-uFZ)^{-1}A_{Rv}\right).
 \end{aligned}
\end{equation}
Moreover, the $u^d$ coefficient of $\Omega_{A;v,w}$ is homogeneous of degree $d$ in the vertex variables.
\end{lemma}
\begin{proof}
Write $M_D(u)=(I-uDZ)^{-1}$. Since $D=F-A_{Rw}A_{wR}/s$, we have
\[
 \begin{aligned}
 (I-uDZ)M_F(u)A_{Rw}
 &=A_{Rw}+\frac us A_{Rw}A_{wR}ZM_F(u)A_{Rw}\\
 &=\frac{\delta(u)}s A_{Rw}.
 \end{aligned}
\]
Hence
\[
 M_D(u)A_{Rw}=\frac{s}{\delta(u)}M_F(u)A_{Rw}.
\]
Combining with $C_{vR}=A_{vR}-a_{vw}A_{wR}/s$, we obtain
\[
 \begin{aligned}
 \frac{a_{vw}}s+\frac us C_{vR}ZM_D(u)A_{Rw}
 &=\frac{a_{vw}}s+
   \frac u{\delta(u)}
    \left(A_{vR}-\frac{a_{vw}}s A_{wR}\right)ZM_F(u)A_{Rw}\\
 &=\frac{a_{vw}}s+
   \frac{uA_{vR}ZM_F(u)A_{Rw}}{\delta(u)}
   -\frac{a_{vw}}s\frac{\delta(u)-s}{\delta(u)}\\
 &=\frac{a_{vw}+uA_{vR}ZM_F(u)A_{Rw}}{\delta(u)}
 =\frac{\beta(u)}{\delta(u)}.
 \end{aligned}
\]
Multiplying by \(\gamma(u)\) gives \eqref{eq15}. The above computation only requires \(s\neq0\). By formal geometric series expansion,
\[
 (I-uDZ)^{-1}=\sum_{h\geq0}u^h(DZ)^h,\qquad
 (I-uFZ)^{-1}=\sum_{h\geq0}u^h(FZ)^h.
\]
In each term of the two factors in \eqref{eq15}, the degree in \(u\) equals the total degree in the vertex variables. This property is preserved under multiplication. Therefore $[u^d]\Omega_{A;v,w}(u)$ is homogeneous of degree \(d\) in the vertex variables, and this property does not depend on the nonnegativity of the matrix entries. From $A\in\operatorname{TP}_2(V)$, \(w\) an endpoint, and Lemma \ref{lem5}, we get $C\in\operatorname{TP}_1$. Hence $C_{vR}$, \(D\), \(F\), and the various row and column blocks in \eqref{eq15} are entrywise nonnegative. Combining $s>0$ with the above formal geometric series expansion, the coefficients of both factors are nonnegative; hence all coefficients of $\Omega_{A;v,w}$ are nonnegative.
\end{proof}

\subsection{Proof of the main theorem}\label{subsec:hook-proof}
We now extract the marked cycle recurrence from Lemma~\ref{lem9} and
combine Lemmas~\ref{lem5} and~\ref{lem10} to prove the nonnegativity of
all marked cycle sums by strong induction.

\begin{proof}[Proof of Theorem~\ref{thm1.1}]
Let $V$ be a finite totally ordered set, $|V|=a+k$, where $a\geq1$, $k\geq0$, and let $A$ be a real matrix indexed by $V$.
For fixed $\sigma\in\mathfrak{S}(V)$, each cycle of length $q\geq a$ contributes \(q\) marked vertices \(v\) satisfying $\ell_v(\sigma)\geq a$; therefore
\[
 \sum_{v\in V}\mathbf{1}_{\{\ell_v(\sigma)\geq a\}}
 =\sum_{q\geq a}q\,m_q(\sigma),
\]
where $\mathbf{1}$ is the indicator function. Substituting this into \eqref{eq3}, multiplying by $\operatorname{wt}_A(\sigma)$, summing over all permutations, and then interchanging the finite sums, we obtain from Definition \eqref{eq4}
\begin{equation}\label{eq5}
 \kappa_{a,k}\phi^{(a,1^k)}[A]
 =\sum_{v\in V}\Psi_v^{a,k}(A).
\end{equation}
Since $\kappa_{a,k}>0$, it suffices to prove that when $A\in\operatorname{TP}_{k+1}(V)$, we have $\Psi_v^{a,k}(A)\geq0$ for every $v\in V$. Let $k\geq1$, take distinct $v,w\in V$, and assume $s=a_{ww}\neq0$. Let \(C\) be the Schur complement with respect to the pivot \(s\), $R=V\setminus\{v,w\}$, $N=A[V\setminus\{v\}]$, and for $T\subseteq R$ write
\[
 A_T=A[(R\setminus T)\cup\{w\}],\qquad
 \eta_T=[u^{|T|}z_T]\Omega_{A;v,w}(u).
\]
Extract the coefficient of $u^{a-1}t^kz_R$ from identity \eqref{eq14}. Since $R\cup\{w\}=V\setminus\{v\}$ and by \eqref{eq12}, the left-hand side is
\[
 [u^{a-1}t^kz_R]\bigl([z_w]_{z_w}K_{A,v}\bigr)
 =[u^{a-1}t^kz_{V\setminus\{v\}}]K_{A,v}
 =\Psi_v^{a,k}(A).
\]
Since the index set of \(C\) is $V\setminus\{w\}=\{v\}\cup R$, its order is $a+(k-1)$, and $k-1\geq0$, by \eqref{eq12} the first term on the right-hand side gives
\[
 [u^{a-1}t^kz_R](tsK_{C,v})
 =s[u^{a-1}t^{k-1}z_R]K_{C,v}
 =s\Psi_v^{a,k-1}(C).
\]
For the second term $u\Omega_{A;v,w}K_{N,w}$, if we select the squarefree term in $\Omega_{A;v,w}$ corresponding to the vertex set $T\subseteq R$, then by the homogeneity obtained in the proof of Lemma \ref{lem10} under the condition $s\neq0$, this term is $\eta_Tu^{|T|}z_T$. By \eqref{eq7}, \(K_{N,w}\) uses $R\setminus T$, and the required degree in \(u\) is
\[
 a-1-1-|T|=a-2-|T|.
\]
Therefore only terms with $|T|\leq a-2$ contribute. Set $a'=a-1-|T|\geq1$. By the locality in Lemma \ref{lem7}, the corresponding coefficient can be extracted on \(A_T\). Since
\[
 |(R\setminus T)\cup\{w\}|=a+k-1-|T|=a'+k,
\]
we have
\[
 [u^{a-2-|T|}t^kz_{R\setminus T}]K_{N,w}
 =\Psi_w^{a',k}(A_T).
\]
Summing over all \(T\) satisfying the condition, we obtain
\begin{equation}\label{eq17}
 \Psi_v^{a,k}(A)=s\Psi_v^{a,k-1}(C)
 +\sum_{\substack{T\subseteq R\\|T|\leq a-2}}
       \eta_T\Psi_w^{a-1-|T|,k}(A_T).
\end{equation}
When $a=1$, the sum on the right-hand side is empty and its value is taken to be \(0\). If moreover $A\in\operatorname{TP}_{k+1}(V)$, \(w\) is an endpoint, and $s>0$, then from $k\geq1$, the preservation of $\operatorname{TP}_{k+1}$ by principal submatrices, and Lemmas \ref{lem5} and \ref{lem10}, we obtain
\[
 C\in\operatorname{TP}_k(V\setminus\{w\}),\qquad
 A_T\in\operatorname{TP}_{k+1}((R\setminus T)\cup\{w\}),\qquad
 \eta_T\geq0.
\]
We now proceed by strong induction on $n=|V|$, proving simultaneously the assertion $\Psi_v^{a,k}(A)\geq0$ for all integer parameter pairs $(a,k)$ with $a\geq1$, $k\geq0$, $a+k=n$, all $n\times n$ matrices $A\in\operatorname{TP}_{k+1}(V)$, and all marked vertices $v\in V$. The index sets of all submatrices and Schur complements inherit the natural order from \(V\). When $k=0$, $a=n$, the permutations in question are all \(n\)-cycles, and their signs cancel with $(-1)^{a-1}$ in the definition. Since $A\in\operatorname{TP}_1$, we have
\[
 \Psi_v^{n,0}(A)=
 \sum_{\substack{\sigma\in\mathfrak{S}(V)\\\sigma\text{ contains exactly one }n\text{-cycle}}}
 \operatorname{wt}_A(\sigma)\geq0.
\]
In particular, when $n=1$, $\Psi_v^{1,0}(A)=a_{vv}\geq0$. Let $n\geq2$, and assume the assertion holds for all matrices of order less than \(n\), the corresponding parameter pairs, and marked vertices.
From the above discussion, it remains to consider $k\geq1$. Fix $a+k=n$ and $v\in V$, and among the two endpoints of \(V\) choose an endpoint \(w\) different from \(v\).
Since $A\in\operatorname{TP}_{k+1}(V)\subseteq\operatorname{TP}_2(V)$, we have $a_{ww}\geq0$.

If $a_{ww}=0$, then by Lemma \ref{lem5}, either the \(w\)-th row or the \(w\)-th column is identically zero. Hence $\operatorname{wt}_A(\sigma)=0$ for all $\sigma\in\mathfrak{S}(V)$, so $\Psi_v^{a,k}(A)=0$.

If $s=a_{ww}>0$, then applying \eqref{eq17}, since $C\in\operatorname{TP}_k(V\setminus\{w\})=\operatorname{TP}_{(k-1)+1}(V\setminus\{w\})$ and the order of \(C\) is $n-1=a+(k-1)$, the induction hypothesis gives $\Psi_v^{a,k-1}(C)\geq0$. For each $T\subseteq R$ with $|T|\leq a-2$, set $a'=a-1-|T|\geq1$. The matrix $A_T\in\operatorname{TP}_{k+1}$ has order $n-1-|T|=a'+k<n$ and contains the marked vertex \(w\). Applying the induction hypothesis again gives $\Psi_w^{a',k}(A_T)\geq0$. Combining $s>0$ and $\eta_T\geq0$, all terms on the right-hand side of \eqref{eq17} are nonnegative, so $\Psi_v^{a,k}(A)\geq0$.
The conclusion still holds when the sum is empty. Therefore, by strong induction, for any $a\geq1$, $k\geq0$, and $A\in\operatorname{TP}_{k+1}(V)$, since $|V|=a+k$, the cycle sum corresponding to every marked vertex $v\in V$ is nonnegative. Take $V=[n]$. From \eqref{eq5} and the nonnegativity of the above marked cycle sums, we obtain
\[
 \kappa_{a,k}\phi^{(a,1^k)}[A]
 =\sum_{v\in[n]}\Psi_v^{a,k}(A)\geq0.
\]
Since $\kappa_{a,k}>0$, we have $\phi^{(a,1^k)}[A]\geq0$.
Since $\ell((a,1^k))=k+1$, when $r\geq k+1$ we have $\operatorname{TP}_r(n)\subseteq\operatorname{TP}_{k+1}(n)$; hence the conclusion holds for any $r\geq\ell(\lambda)$.
\end{proof}

\noindent{\bf Declaration of Competing Interest}\\
{The authors have no conflicts of interest to disclose.}

\noindent{\bf Data Availability}\\
{Data sharing is not applicable as no new data were created or analyzed.}

\noindent{\bf Acknowledgements:} This work is supported by the National Natural Science Foundation of China (Nos. 12261071,  12571019),  and Natural Science Foundation of Qinghai Province (No. 2025-ZJ-902T).


\begin{thebibliography}{abcdsfgh}
\bibitem{cle1}
S. Clearman, B. Shelton  and M. Skandera,
Path tableaux and combinatorial interpretations of immanants for class functions on $S_n$,
\textit{Discrete Math. Theor. Comput. Sci. Proc.} AO (2011), 233--244.

\bibitem{chan}
 O. Chan, T. Lam, Immanant inequalities for Laplacians of trees, \textit{SIAM J. Matrix Anal. Appl.} 21 (1999) 129--144.

\bibitem{don1} 
X. Dong, T. Wu, H. Lai, Some immanantal inequalities and equalities for linear combination matrices of (di)graphs, \textit{Linear Multilinear Algebra}  74 (2026) 1747--1778.

\bibitem{don2} 
X. Dong, T. Wu, Immanantal polynomials of the linear combination matrices of graphs,  http://arxiv.org/abs/2604.04489.

\bibitem{gou1}
I. Goulden and  D. Jackson, Immanants of combinatorial matrices, \textit{J. Algebra} {\bf 148} (1992), 305--324.

\bibitem{gou2}
I. Goulden and  D. Jackson,   Immanants, Schur functions, and the MacMahon master theorem, \textit{Proc. Am. Math. Soc.} {\bf 115} (1992), 605--612.

\bibitem{hai}
M. Haiman, Hecke algebra characters and immanant conjectures, \textit{J. Amer. Math. Soc.} {\bf 6} (1993), 569--595.

\bibitem{hor1}
R. A. Horn  and  F. Zhang, Basic properties of the Schur complement,
in: F. Zhang (Ed.), \textit{The Schur Complement and Its Applications},
Numerical Methods and Algorithms, vol. 4, Springer, 2005, pp. 17--46.

\bibitem{les1}
N. R. T. Lesnevich, Hook-shape immanant characters from Stanley--Stembridge characters,
\textit{Algebr. Comb.} {\bf 7} (2024), 137--157.

\bibitem{lit} 
D. E. Littlewood, The Theory of Group Characters and Matrix Representations of Groups, Oxford Univ. Press, London, 1950.

\bibitem{sch} I. Schur, \"{U}ber endliche Gruppen und Hermitesche Formen, \textit{Math. Z.} {\bf 1} (1918), 184--207.

\bibitem{ska1}
M. Skandera, Hook immanantal inequalities for totally nonnegative matrices,  arXiv:2510.00327, 2025.

\bibitem{ste1}
J. R. Stembridge, Some conjectures for immanants, \textit{Canad. J. Math.} {\bf 44} (1992), 1079--1099.

\bibitem{yu} G. Yu, H. Qu, The coefficients of the immanantal polynomial, \textit{Appl. Math. Comput.} 339 (2018) 38--44.

\end{thebibliography}
\end{document}